\documentclass[12pt]{article}
\usepackage[top=3cm,bottom=3cm,left=3cm,right=3cm]{geometry}

\usepackage{amsmath,amsthm,amssymb,authblk,enumerate}

\def\d{\delta}
\def\g{\gamma}
\def\s{\sigma}
\def\t{\tilde}

\def\R{\mathbf R}
\def\P{\mathbf P}

\DeclareMathOperator\spa{{span}}

\numberwithin{equation}{section}
\newtheorem{thm}{Theorem}[section]
\newtheorem{lem}[thm]{Lemma}

\begin{document}
\title{Bivariate Lagrange interpolation at the checkerboard nodes}

\author[1]{Lihua Cao}
\author[2]{Srijana Ghimire}
\author[2]{Xiang-Sheng Wang\thanks{Corresponding author. Email: xswang@louisiana.edu}}

\affil[1]{College of Mathematics and Statistics, Shenzhen University, Shenzhen, Guangdong 518060, China}
\affil[2]{Department of Mathematics, University of Louisiana at Lafayette, Lafayette, LA 70503, USA}

\maketitle

\begin{abstract}
In this paper, we derive an explicit formula for the bivariate Lagrange basis polynomials of a general set of checkerboard nodes. This formula generalizes existing results of bivariate Lagrange basis polynomials at the Padua nodes, Chebyshev nodes, Morrow-Patterson nodes, and Geronimus nodes. We also construct a subspace spanned by linearly independent bivariate vanishing polynomials that vanish at the checkerboard nodes and prove the uniqueness of the set of bivariate Lagrange basis polynomials in the quotient space defined as the space of bivariate polynomials with a certain degree by the subspace of bivariate vanishing polynomials.
\end{abstract}


\bigskip


\newpage

\section{Introduction}

Given $x_0>x_1>\cdots>x_n$ and $y_0>y_1>\cdots>y_{n+\s}$ where $n$ and $\s$ are nonnegative integers, we define
a rectangular set of nodes:
\begin{equation}
  S=\{(x_r,y_u):~0\le r\le n,~0\le u\le n+\s\},
\end{equation}
which consists of $(n+1)(n+\s+1)$ distinct points in $\R^2$.
The set $S$ can be divided into two checkerboard sets $S_0$ and $S_1$ such that $(x_r,y_u)\in S_0$ if and only if $r+u$ is even while $(x_r,y_u)\in S_1$ if and only if $r+u$ is odd.
Our objective is to develop existence and uniqueness theory of bivariate Lagrange basis polynomials for the checkerboard set $S_\tau$ with $\tau=0$ or $\tau=1$. The special case $\s=0$ was considered in \cite{Harris18JCAM}. The special case $\s=1$ was studied in \cite{Harris15JAT,Harris21PAMS}.
If $\s=1$, $x_r=\cos(r\pi/n)$ and $y_u=\cos[u\pi/(n+1)]$, then $S_\tau$ is the set of Padua points and the corresponding set of bivariate Lagrange basis polynomials is unique in $\P_n(x,y)$; see \cite{Bos06JAT,Bos07NM}.
In \cite{Xu96JAT}, Xu derived bivariate Lagrange basis polynomials when $\s=0$ and $x_k=y_k=\cos[(2k-1)\pi/(2n)]$ are zeros of Chebyshev polynomial of first kind $T_n$; see also \cite{Bojanov97JCAM,Harris10PAMS}.
The checkerboard nodes $S_\tau$ also generalize the Morrow-Patterson nodes \cite{Morrow78SINA} and Geronimus nodes \cite{Harris15JAT}.
In \cite{Harris21PAMS}, several formulas of bivariate Lagrange basis polynomials for the cases $\s=2,3,4,5$ were proposed as conjectures.
In this paper, we will derive a general formula of bivariate Lagrange basis polynomials for any nonnegative integer $\s$.
This formula generalizes the aforementioned results of bivariate Lagrange basis polynomials at the Padua nodes, Chebyshev nodes, Morrow-Patterson nodes and Geronimus nodes. Especially, the conjectures in \cite{Harris21PAMS} for the cases $\s=2,3,4,5$ are proved.
Moreover, we will prove that the set of bivariate Lagrange basis polynomials is unique in a certain quotient space of bivariate polynomials.

Let $\P_d(x,y)$ be the linear space of bivariate polynomials of degree no more than $d$, which can be generated by the monomials $x^jy^k$ with $j+k\le d$.
It is easily seen that the dimension of $\P_d(x,y)$ is $(d+1)(d+2)/2$.
If $f_1(x,y),\cdots,f_M(x,y)$ are linearly independent polynomials in $\P_d(x,y)$; namely, $c_1f_1(x,y)+\cdots+c_Mf_M(x,y)=0$ for all $(x,y)\in\R^2$ implies $c_1=\cdots=c_M=0$, then we can define the quotient space $\P_d(x,y)/\{f_1(x,y),\cdots,f_M(x,y)\}$ in the sense that two polynomials in this quotient space are identical if and only if their difference can be expressed as a linear combination of $f_1(x,y),\cdots,f_M(x,y)$. Clearly, the dimension of $\P_d(x,y)/\{f_1(x,y),\cdots,f_M(x,y)\}$ is $(d+1)(d+2)/2-M$.

Given a set of nodes $(x_1,y_1),\cdots,(x_N,y_N)\in\R^2$, we say $\{L_1(x,y),\cdots,L_N(x,y)\}\subset\P_d(x,y)$ is
a set of bivariate Lagrange basis polynomials if $L_j(x_k,y_k)=0$ for $1\le j\neq k\le N$ and $L_k(x_k,y_k)=1$ for $1\le k\le N$.
For convenience, we also define the {\it bivariate vanishing polynomial} as a bivariate polynomial $f(x,y)\in\P_d(x,y)$ that vanishes at all of the given nodes; namely, $f(x_k,y_k)=0$ for all $1\le k\le N$.

The rest of this paper is organized as follows. In Section 2, we state some preliminary results on one-to-one map between a sequence of univariate nodes and a sequence of difference equations. We also give a necessary and sufficient condition for the uniqueness of bivariate Lagrange basis polynomials.
In Section 3, we construct the bivariate Lagrange basis polynomials for $S_\tau$ with general $\s$.
In Section 4, we prove the uniqueness of the bivariate Lagrange basis polynomials for $S_\tau$ in a certain quotient space of bivariate polynomials.

\section{Preliminary results}

We first rephrase the results in \cite[Lemma 2]{Harris18JCAM} and \cite[Theorem A.1 \& Lemma A.2]{Harris21PAMS} as the following lemma.
\begin{lem}\label{lem-existence}
  Given any $x_0>x_1>\cdots>x_n$, there exists a sequence of orthogonal polynomials $\{p_k(x)\}_{k=0}^n$ determined by two sequences $\{a_k\}_{k=0}^{n-1}$ and $\{b_k\}_{k=0}^{n-1}$ such that $p_0(x)=1$, $p_1(x)=a_0x+b_0$, and
  \begin{equation}
    p_{k+1}(x)+p_{k-1}(x)=(a_kx+b_k)p_k(x),
  \end{equation}
  for $1\le k\le n-1$, and the following properties hold.
  \begin{enumerate}
    \item (positivity) $a_k>0$ for all $0\le k\le n-1$.
    \item (reflection) $a_k=a_{n-k}$ and $b_k=b_{n-k}$ for all $1\le k\le n-1$.
    \item (alternation) $p_{n-k}(x_j)=(-1)^jp_k(x_j)$ for all $0\le j,k\le n$.
  \end{enumerate}
\end{lem}
It is nature to ask whether the map from the set of distinct nodes $X=\{x_k\}_{k=0}^n$ to the set of coefficients $(A,B)=\{(a_k,b_k)\}_{k=0}^{n-1}$ satisfying the positivity and reflection conditions is invertible and unique. It is easy to show that the alternation condition implies the invertibility of the map.
When $n$ is odd, the dimension of $(A,B)$ is $n+1$, which suggests that the map might be unique.
However, when $n$ is even, the dimension of $(A,B)$ becomes $n+2$, which indicates that there is one more degree of freedom in $(A,B)$.
We shall prove that this additional degree of freedom can be removed by a normalization condition $a_0=1$.

\begin{thm}\label{thm-inv}
  Given any sequences $\{a_k\}_{k=0}^{n-1}$ and $\{b_k\}_{k=0}^{n-1}$ satisfying the positivity and reflection conditions; namely, $a_k>0$ for all $0\le k\le n-1$ and $a_k-a_{n-k}=b_k-b_{n-k}=0$ for all $1\le k\le n-1$, there exists a unique set of nodes $x_0>x_1>\cdots>x_n$ satisfying the alternation condition $p_{n-k}(x_j)=(-1)^jp_k(x_j)$ for all $0\le j,k\le n$, where $\{p_k(x)\}_{k=0}^n$ is the sequence of orthogonal polynomials determined by the difference equation $p_{k+1}(x)+p_{k-1}(x)=(a_kx+b_k)p_k(x)$ for $1\le k\le n-1$ with initial conditions $p_0(x)=1$ and $p_1(x)=a_0x+b_0$.
\end{thm}
\begin{proof}
  When $n=2m-1$ is odd, we denote by $u_1>\cdots>u_m$ the zeros of $p_m(x)$ and $v_1>v_2>\cdots>v_{m-1}$ the zeros of $p_{m-1}(x)$.
  By alternation property of zeros of orthogonal polynomials, we have $u_1>v_1>u_2>v_2>\cdots>v_{m-1}>u_m$.
  Since both $p_m(x)$ and $p_{m-1}(x)$ have positive leading coefficients, the difference function $p_m(x)-p_{m-1}(x)$ has at least one zero at each of the intervals $(u_1,\infty),(u_2,v_1),\cdots,(u_m,v_{m-1})$ because the difference takes opposite signs when the variable $x$ approaches the two ends of each interval. Similarly, the sum function $p_m(x)+p_{m-1}(x)$ has at least one zero at each of the intervals $(v_1,u_1),\cdots,(v_{m-1},u_{m-1}),(-\infty,u_m)$.
  Thus, we can order the zeros of $p_m(x)-p_{m-1}(x)$ and $p_m(x)+p_{m-1}(x)$ as $x_0>x_1>\cdots>x_n$ such that $p_m(x_j)-(-1)^jp_{m-1}(x_j)=0$ for all $0\le j\le n$.
  Actually, we have $x_0>u_1>x_1>v_1>x_2>v_2>\cdots>v_{m-1}>x_{n-1}>u_m>x_n$.
  It then follows from the difference equation and the reflection condition that $p_{n-k}(x_j)=(-1)^jp_k(x_j)$ for all $0\le j,k\le n$.

  When $n=2m$ is even, the zeros of $p_m(x)$ divide the real line into $m+1$ intervals, each of which contains at least one zero of $p_{m+1}(x)-p_{m-1}(x)$, thanks to the alternation property of orthogonal polynomials.
  Hence, we can order the zeros of $p_m(x)$ and $p_{m+1}(x)-p_{m-1}(x)$ as $x_0>x_1>\cdots>x_n$ such that $p_m(x_j)=0$ for all odd $j=1,3,\cdots,n-1$ and $p_{m+1}(x_j)=p_{m-1}(x_j)$ for all even $j=0,2,\cdots,n$.
  It then follows from the difference equation and the reflection condition that $p_{n-k}(x_j)=(-1)^jp_k(x_j)$ for all $0\le j,k\le n$.
  This completes the proof.
\end{proof}

\begin{thm}
  Given any $x_0>x_1>\cdots>x_n$, if there are two sets of orthogonal polynomials $\{p_k(x)\}_{k=0}^n$ and $\{\t p_k(x)\}_{k=0}^n$, which are determined by two sets of coefficients $(A,B)=\{(a_k,b_k)\}_{k=0}^{n-1}$ and $(\t A,\t B)=\{(\t a_k,\t b_k)\}_{k=0}^{n-1}$, such that the difference equation and three properties (positivity, reflection and alternation) in Lemma \ref{lem-existence} are satisfied, then we have the following results depending on whether $n$ is odd or even.
  \begin{enumerate}
    \item If $n$ is odd, then $\t a_k=a_k$ and $\t b_k=b_k$ for $0\le k\le n-1$.
    \item If $n$ is even, then there exists a positive constant $\g$ such that $\t a_k/a_k=\t b_k/b_k=\g$ for even $k=0,2,\cdots,n$ and $a_k/\t a_k=b_k/\t b_k=\g$ for odd $k=1,3,\cdots,n-1$.
  \end{enumerate}
\end{thm}
\begin{proof}
  When $n=2m-1$ is odd, it follows from the alternation condition that the zeros of $p_m(x)+p_{m-1}(x)$ and $\t p_m(x)+\t p_{m-1}(x)$ are the same while the zeros of $p_m(x)-p_{m-1}(x)$ and $\t p_m(x)-\t p_{m-1}(x)$ are the same. The positivity condition implies that there exist two positive constants $\g_1$ and $\g_2$ such that $\t p_m(x)+\t p_{m-1}(x)=\g_1[p_m(x)+p_{m-1}(x)]$ and $\t p_m(x)-\t p_{m-1}(x)=\g_2[p_m(x)-p_{m-1}(x)]$.
  Comparing the leading coefficients yields $\g_1=\g_2$. A simple combination of these two equations then gives $\t p_m(x)=\g_1 p_m(x)$ and $\t p_{m-1}(x)=\g_1 p_{m-1}(x)$. In view of the difference equations $\t p_m(x)+\t p_{m-2}(x)=(\t a_{m-1}x+\t b_{m-1})\t p_{m-1}(x)$ and $p_m(x)+p_{m-2}(x)=(a_{m-1}x+b_{m-1})\t p_{m-1}(x)$, we obtain $\t a_{m-1}=a_{m-1}$, $\t b_{m-1}=b_{m-1}$ and $\t p_{m-2}(x)=\g_1 p_{m-2}(x)$. Repeating this argument implies $\t a_k=a_k$ and $\t b_k=b_k$ for $0\le k\le m-1$. Moreover, $\t p_0(x)=\g_1p_0(x)$, which yields $\g_1=1$. The reflection condition then gives $\t a_k=a_k$ and $\t b_k=b_k$ for $0\le k\le n-1$.

  When $n=2m$ is even, it follows from the alternation condition that the zeros of $p_m(x)$ and $\t p_m(x)$ are the same while the zeros of $p_{m+1}(x)-p_{m-1}(x)$ and $\t p_{m+1}(x)-\t p_{m-1}(x)$ are the same. The positivity condition implies that there exist two positive constants $\g_1$ and $\g_2$ such that $\t p_m(x)=\g_1 p_m(x)$ and $\t p_{m+1}(x)-\t p_{m-1}(x)=\g_2[p_{m+1}(x)-p_{m-1}(x)]$.
  On account of the difference equations $\t p_{m+1}(x)+\t p_{m-1}(x)=(\t a_mx+\t b_m)\t p_m(x)$ and $p_{m+1}(x)+p_{m-1}(x)=(a_mx+b_m)p_m(x)$, we obtain
  $\g_2a_m=\g_1\t a_m$, $\g_2b_m=\g_1\t b_m$, and $\t p_{m+1}(x)+\t p_{m-1}(x)=\g_2[p_{m+1}(x)+p_{m-1}(x)]$.
  Consequently, $\t p_{m+1}(x)=\g_2p_{m+1}(x)$ and $\t p_{m-1}(x)=\g_2p_{m-1}(x)$.
  It then follows from the difference equations $\t p_m(x)+\t p_{m-2}(x)=(\t a_{m-1}x+\t b_{m-1})\t p_{m-1}(x)$ and $p_m(x)+p_{m-2}(x)=(a_{m-1}x+b_{m-1})\t p_{m-1}(x)$ that $\g_1a_{m-1}=\g_2\t a_{m-1}$, $\g_1b_{m-1}=\g_2\t b_{m-1}$ and $\t p_{m-2}(x)=\g_1p_{m-2}(x)$.
  Repeating this argument gives
  $$\g_1a_{m-j}=\g_2\t a_{m-j},~\g_1b_{m-j}=\g_2\t b_{m-j},~\t p_{m-j}(x)=\g_2p_{m-j}(x)$$
  for odd $j\le m$, and
  $$\g_2a_{m-j}=\g_1\t a_{m-j},~\g_2b_{m-j}=\g_1\t b_{m-j},~\t p_{m-j}(x)=\g_1p_{m-j}(x)$$
  for even $j\le m$. Since $\t p_0(x)=p_0(x)=1$, either $\g_1=1$ (when $m$ is even) or $\g_2=1$ (when $m$ is odd).
  We denote $\g=\g_2$ if $m$ is even and $\g=\g_1$ if $m$ is odd. It then follows that
  $$\t a_{2j}=\g a_{2j},~\t b_{2j}=\g b_{2j},~~\t p_{2j}(x)=p_{2j}(x),$$
  for $j=0,\cdots,m$, and
  $$\g\t a_{2j+1}=a_{2j+1},~~\g\t b_{2j+1}=b_{2j+1},~~\t p_{2j+1}(x)=\g p_{2j}(x),$$
  for $j=0,\cdots,m-1$. This completes the proof.
\end{proof}

For the univariate case, the set of Lagrange basis polynomials for any set of distinct points exists and is uniquely determined because the corresponding Vandermonde matrix is invertible. The following theorem gives criteria for uniqueness of bivariate Lagrange basis polynomials.

\begin{thm}\label{thm-uniqueness}
  Given any distinct points $(x_1,y_1),\cdots,(x_N,y_N)\in\R^2$ and any positive integer $d$ such that $(d+1)(d+2)/2\ge N$, there exist at least $M=(d+1)(d+2)/2-N$ linear independent bivariate vanishing polynomials, denoted by $f_1(x,y),\cdots,f_M(x,y)$, in $\P_d(x,y)$.
  Let
  \begin{align}
    V=\begin{pmatrix}
      1&x_1&\cdots&x_1^d&y_1&x_1y_1&\cdots&x_1^{d-1}y_1&\cdots&y_1^d\\
      \vdots&\vdots&&\vdots&\vdots&\vdots&&\vdots&&\vdots\\
      1&x_N&\cdots&x_N^d&y_N&x_Ny_N&\cdots&x_N^{d-1}y_N&\cdots&y_N^d\\
    \end{pmatrix}
  \end{align}
  be the bivariate Vandermonde matrix of dimension $N$ by $(d+1)(d+2)/2$.
  The following statements are equivalent.
  \begin{enumerate}[(i)]
    \item There exists a unique set of bivariate Lagrange interpolation polynomials in the quotient space $\P_d(x,y)/\{f_1(x,y),\cdots,f_M(x,y)\}$.
    \item There exists a set of bivariate Lagrange interpolation polynomials in $\P_d(x,y)$.
    \item Any bivariate vanishing polynomial in $\P_d(x,y)$ can be expressed as a linear combination of $f_1(x,y),\cdots,f_M(x,y)$.
    \item The rank of $V$ is $N$.
  \end{enumerate}
\end{thm}
\begin{proof}
  The coefficients of any bivariate vanishing polynomial in $\P_d(x,y)$ corresponds a vector $z\in\R^{(d+1)(d+2)/2}$ satisfying $Vz=0$.
  The existence of $f_1(x,y),\cdots, f_M(x,y)$ follows from the fact that the rank of $V$ is no more than $N$.
  Moreover, we have (iii) $\Leftrightarrow$ (iv).
  Any set of bivariate Lagrange interpolation in $\P_d(x,y)$ can be represented by the matrix $L$ of dimension $(d+1)(d+2)/2$ by $N$ such that $VL$ is the identity matrix in $\R^{N\times N}$. Hence, (ii) $\Leftrightarrow$ (iv).
  It is obvious that (i) $\Rightarrow$ (ii).
  Finally, coupling (ii) and (iii) gives (i).
  The proof is complete.
\end{proof}

\section{Existence of bivariate Lagrange basis polynomials}
Given $x_0>x_1>\cdots>x_n$ and $y_0>y_1>\cdots>y_{n+\s}$ where $n$ and $\s$ are nonnegative integers, we define two sets of checkerboard nodes
\begin{align}
  S_0&=\{(x_r,y_u):~0\le r\le n,~0\le u\le n+\s,~r+u~\text{even}\},\label{S0}\\
  S_1&=\{(x_r,y_u):~0\le r\le n,~0\le u\le n+\s,~r+u~\text{odd}\},\label{S1}
\end{align}
which consist of $N_0$ and $N_1$ nodes, respectively.
It is easily seen that $N_0+N_1=(n+1)(n+\s+1)$.
Moreover, we have $N_0-N_1=1$ and $N_\tau=[(n+1)(n+\s+1)+1]/2-\tau$ if both $n$ and $\s$ are even, while $N_0=N_1=(n+1)(n+\s+1)/2$ if either $n$ or $\s$ is odd.
We need to find a set of bivariate Lagrange basis polynomials for $S_\tau$ with $\tau=0$ or $\tau=1$.
According to Lemma \ref{lem-existence}, there exist orthogonal polynomials $\{p_j(x)\}_{j=0}^n$ and $\{q_k(y)\}_{k=0}^{n+\s}$ such that
$p_0(x)=1$, $p_1(x)=a_0x+b_0$, $q_0(y)=1$, $q_1(x)=c_0y+d_0$, and
\begin{align}
  p_{j+1}(x)+p_{j-1}(x)&=(a_jx+b_j)p_j(x),~~1\le j\le n-1,\\
  q_{k+1}(y)+q_{k-1}(y)&=(c_ky+d_k)q_k(y),~~1\le k\le n+\s-1,
\end{align}
where $a_j>0$ for $0\le j\le n-1$, and $c_k>0$ for $0\le k\le n+\s-1$, and
\begin{align}
  a_j=a_{n-j}, b_j=b_{n-j},&~~1\le j\le n-1,\label{a}\\
  c_k=c_{n+\s-k}, d_k=d_{n+\s-k},&~~1\le k\le n+\s-1,\label{c}\\
  p_{n-j}(x_r)=(-1)^rp_j(x_r),&~~0\le j,r\le n,\label{p}\\
  q_{n+\s-k}(y_u)=(-1)^uq_k(y_u),&~~0\le k,u\le n+\s.\label{q}
\end{align}
For any $(x_r,y_u)\in S_\tau$ and $(x_s,y_v)\in S_\tau$, where $\tau$ is either $0$ or $1$, it is easily seen that $r+u+s+v$ is even.
Hence, we have from the above two equations
\begin{align}\label{pq}
  p_j(x_r)p_i(x_s)q_k(y_u)q_l(y_v)=p_{n-j}(x_r)p_{n-i}(x_s)q_{n+\s-k}(y_u)q_{n+\s-l}(y_v),
\end{align}
for all $0\le j,i\le n$ and $0\le k,l\le n+\s$.
Moreover, we have the following Christoffel-Darboux formulas:
\begin{align}
  (x_r-x_s)\sum_{j=0}^ia_jp_j(x_r)p_j(x_s)&=p_{i+1}(x_r)p_i(x_s)-p_i(x_r)p_{i+1}(x_s),\label{CD-x}\\
  (y_u-y_v)\sum_{k=0}^lc_kq_k(y_u)q_k(y_v)&=q_{l+1}(y_u)q_l(y_v)-q_l(y_u)q_{l+1}(y_v),\label{CD-y}
\end{align}
for $0\le i\le n-1$ and $0\le l\le n+\s-1$.
Given any integers $0\le\d\le\s-1$, $0\le s\le n$ and $0\le v\le n+\s$, we define the bivariate polynomial
\begin{equation}\label{Kd}
  K_\d(x,y;x_s,y_v)=\sum_{j=0}^{n-1}a_jp_j(x)p_j(x_s)\sum_{k=0}^{n-j+\d}c_kq_k(y)q_k(y_v)\in\P_{n+\d}(x,y).
\end{equation}
It is readily seen from \eqref{pq} and \eqref{CD-y} that
\begin{align*}
  &(y_u-y_v)K_\d(x_r,y_u;x_s,y_v)
  \\=&\sum_{j=0}^{n-1}a_jp_j(x_r)p_j(x_s)[q_{n-j+\d+1}(y_u)q_{n-j+\d}(y_v)-q_{n-j+\d}(y_u)q_{n-j+\d+1}(y_v)]
  \\=&\sum_{j=0}^{n-1}a_jp_{n-j}(x_r)p_{n-j}(x_s)[q_{j+\s-\d-1}(y_u)q_{j+\s-\d}(y_v)-q_{j+\s-\d}(y_u)q_{j+\s-\d-1}(y_v)]
  \\=&\sum_{j=1}^na_{n-j}p_j(x_r)p_j(x_s)[q_{n-j+\s-\d-1}(y_u)q_{n-j+\s-\d}(y_v)-q_{n-j+\s-\d}(y_u)q_{n-j+\s-\d-1}(y_v)],
\end{align*}
and
\begin{align*}
  &(y_u-y_v)K_{\s-\d-1}(x_r,y_u;x_s,y_v)
  \\=&\sum_{j=0}^{n-1}a_jp_j(x_r)p_j(x_s)[q_{n-j+\s-\d}(y_u)q_{n-j+\s-\d-1}(y_v)-q_{n-j+\s-\d-1}(y_u)q_{n-j+\s-\d}(y_v)].
\end{align*}
Adding the above two equations and making use of \eqref{a}, \eqref{pq} and \eqref{CD-y} yield
\begin{align*}
  &(y_u-y_v)[K_\d(x_r,y_u;x_s,y_v)+K_{\s-\d-1}(x_r,y_u;x_s,y_v)]
  \\=&a_0p_n(x_r)p_n(x_s)[q_{\s-\d-1}(y_u)q_{\s-\d}(y_v)-q_{\s-\d}(y_u)q_{\s-\d-1}(y_v)]
  \\&~+a_0p_0(x_r)p_0(x_s)[q_{n+\s-\d}(y_u)q_{n+\s-\d-1}(y_v)-q_{n+\s-\d-1}(y_u)q_{n+\s-\d}(y_v)]
  \\=&a_0p_n(x_r)p_n(x_s)[q_{\s-\d-1}(y_u)q_{\s-\d}(y_v)-q_{\s-\d}(y_u)q_{\s-\d-1}(y_v)]
  \\&~+a_0p_n(x_r)p_n(x_s)[q_{\d}(y_u)q_{\d+1}(y_v)-q_{\d+1}(y_u)q_{\d}(y_v)]
  \\=&-(y_u-y_v)a_0p_n(x_r)p_n(x_s)[\sum_{k=0}^{\s-\d-1}c_kq_k(y_u)q_k(y_v)+\sum_{k=0}^{\d}c_kq_k(y_u)q_k(y_v)],
\end{align*}
which can be written as
\begin{equation}\label{Gy}
  (y_u-y_v)[K_\d(x_r,y_u;x_s,y_v)+K_{\s-\d-1}(x_r,y_u;x_s,y_v)+L(x_r,y_u;x_s,y_v)]=0,
\end{equation}
where
\begin{equation}\label{J}
  J(x,y;x_s,y_v)=a_0p_n(x)p_n(x_s)[\sum_{k=0}^{\s-\d-1}c_kq_k(y)q_k(y_v)+\sum_{k=0}^{\d}c_kq_k(y)q_k(y_v)].
\end{equation}
Interchanging the double sum in \eqref{Kd} gives another expression
\begin{equation*}
  K_\d(x,y;x_s,y_v)=\sum_{k=0}^{n+\d}c_kq_k(y)q_k(y_v)\sum_{j=0}^{\min\{n-1,n+\d-k\}}a_jp_j(x)p_j(x_s).
\end{equation*}
It is readily seen from \eqref{c}, \eqref{pq} and \eqref{CD-x} that
\begin{align*}
  &(x_r-x_s)K_\d(x_r,y_u;x_s,y_v)
  \\=&\sum_{k=1+\d}^{n+\d}c_kq_k(y_u)q_k(y_v)[p_{n-k+\d+1}(x_r)p_{n-k+\d}(x_s)-p_{n-k+\d}(x_r)p_{n-k+\d+1}(x_s)]
  \\&+\sum_{k=0}^{\d}c_kq_k(y_u)q_k(y_v)[p_{n}(x_r)p_{n-1}(x_s)-p_{n-1}(x_r)p_{n}(x_s)]
  \\=:&A_\d+B_\d,
\end{align*}
where
\begin{align*}
  A_\d=&\sum_{k=1+\d}^{n+\d}c_{n+\s-k}q_{n+\s-k}(y_u)q_{n+\s-k}(y_v)[p_{k-\d-1}(x_r)p_{k-\d}(x_s)-p_{k-\d}(x_r)p_{k-\d-1}(x_s)]
  \\=&\sum_{k=\s-\d}^{n+\s-\d-1}c_kq_k(y_u)q_k(y_v)[p_{n-k+\s-\d-1}(x_r)p_{n-k+\s-\d}(x_s)-p_{n-k+\s-\d}(x_r)p_{n-k+\s-\d-1}(x_s)]
  \\=&-A_{\s-\d-1},
\end{align*}
and
\begin{align*}
  B_\d=&\sum_{k=0}^{\d}c_kq_{n-k+\s}(y_u)q_{n-k+\s}(y_v)[p_0(x_r)p_1(x_s)-p_1(x_r)p_0(x_s)]
  \\=&-(x_r-x_s)\sum_{k=0}^{\d}c_kq_{n-k+\s}(y_u)q_{n-k+\s}(y_v)a_0p_0(x_r)p_0(x_s)
  \\=&-(x_r-x_s)\sum_{k=0}^{\d}c_kq_k(y_u)q_k(y_v)a_0p_n(x_r)p_n(x_s).
\end{align*}
Hence,
$$(x_r-x_s)[K_\d(x_r,y_u;x_s,y_v)+K_{\s-\d-1}(x_r,y_u;x_s,y_v)]=B_\d+B_{\s-\d-1}.$$
Recall the definition of $J$ in \eqref{J}, we obtain
\begin{equation}\label{Gx}
  (x_r-x_s)[K_\d(x_r,y_u;x_s,y_v)+K_{\s-\d-1}(x_r,y_u;x_s,y_v)+J(x_r,y_u;x_s,y_v)]=0.
\end{equation}
Finally, we choose $\d=\lfloor\s/2\rfloor$ and define the bivariate polynomial
\begin{align}\label{G}
  G(x,y;x_s,y_v)=&K_\d(x,y;x_s,y_v)+K_{\s-\d-1}(x,y;x_s,y_v)+J(x,y;x_s,y_v)
  \notag\\=&\sum_{j=0}^{n-1}a_jp_j(x)p_j(x_s)[\sum_{k=0}^{n-j+\d}c_kq_k(y)q_k(y_v)
  +\sum_{k=0}^{n-j+\s-\d-1}c_kq_k(y)q_k(y_v)]
  \notag\\&+a_0p_n(x)p_n(x_s)[\sum_{k=0}^{\s-\d-1}c_kq_k(y)q_k(y_v)+\sum_{k=0}^{\d}c_kq_k(y)q_k(y_v)].
\end{align}
It is obvious that $G\in\P_{n+\d}(x,y)$ and $G(x_s,y_v;x_s,y_v)>0$. Coupling \eqref{Gy} and \eqref{Gx} gives $G(x_r,y_u;x_s,y_v)=0$ if either $x_r\neq x_s$ or $y_u\neq y_v$.
We summarize our results in the following theorem.
\begin{thm}\label{thm-existence}
  Given any $x_0>x_1>\cdots>x_n$ and $y_0>y_1>\cdots>y_{n+\s}$ where $n$ and $\s$ are nonnegative integers.
  Let $S_\tau$ with either $\tau=0$ or $\tau=1$ be defined in \eqref{S0} or \eqref{S1}.
  Set $\d=\lfloor\s/2\rfloor$.
  There exists a set of bivariate Lagrange basis polynomials in $\P_{n+\d}(x,y)$ for $S_\tau$ which can be defined as
  \begin{equation}\label{L}
    L(x,y;x_s,y_v)=G(x,y;x_s,y_v)/G(x_s,y_v;x_s,y_v),
  \end{equation}
  for each $(x_s,y_v)\in S_\tau$,
  where $G$ is defined in \eqref{G}.
\end{thm}

\section{Uniqueness of bivariate Lagrange basis polynomials}
We use the same notations as in the previous section.
To prove that the set of bivariate Lagrange basis polynomials constructed in \eqref{L} is unique in a certain quotient space of $\P_{n+\d}(x,y)$,
we only need to find $M=(n+\d+1)(n+\d+2)/2-N_\tau$ linearly independent bivariate vanishing polynomials in $\P_{n+\d}(x,y)$, where $\d=\lfloor\s/2\rfloor$ and $N_\tau$ is the number of nodes in $S_\tau$ with $\tau=0$ or $\tau=1$. In other words, we shall construct a linear subspace $Q$ of bivariate vanishing polynomials in $\P_{n+\d}(x,y)$ with dimension $M$ and then apply Theorem \ref{thm-uniqueness}.
First, we introduce the linear subspace of bivariate vanishing polynomials:
\begin{equation}\label{V}
  V=\spa\{(x-x_0)\cdots(x-x_n)x^jy^k,~~j\ge0,~k\ge0,~j+k\le \d-1\}.
\end{equation}
It is obvious that $V$ is a subspace of $\P_{n+\d}(x,y)$ with dimension $\d(\d+1)/2$.
We shall consider the following three cases respectively.

\begin{enumerate}[{Case} I.]
\item $\s=2\d+1$ is odd.

We have $N_0=N_1=(n+1)(n+\s+1)/2$ and
$$M={(n+\d+1)(n+\d+2)\over2}-{(n+1)(n+\s+1)\over2}={\d(\d+1)\over2}$$
for either $\tau=0$ or $\tau=1$.
We simply set
\begin{equation}\label{Q1}
  Q=V.
\end{equation}

\item $\s=2\d$ is even and $n=2m-1$ is odd.

We have $N_0=N_1=(n+1)(n+\s+1)/2$ and
$$M={(n+\d+1)(n+\d+2)\over2}-{(n+1)(n+\s+1)\over2}={\d(\d+1)\over2}+m$$
for either $\tau=0$ or $\tau=1$.
Note from \eqref{p} and \eqref{q} that
$$p_{n-j}(x_r)q_{j+\d}(y_u)=(-1)^{r+u}p_j(x_r)q_{n+\d-j}(y_u)=(-1)^\tau p_j(x_r)q_{n+\d-j}(y_u),$$
for any $(x_r,y_u)\in S_\tau$.
We define
\begin{equation}\label{Q2}
  Q=V+\spa\{p_{n-j}(x)q_{j+\d}(y)-(-1)^\tau p_j(x)q_{n+\d-j}(y),~~0\le j\le m-1\},
\end{equation}
which is a subspace of bivariate vanishing polynomials in $\P_{n+\d}(x,y)$ with dimension $M=\d(\d+1)/2+m$.

\item $\s=2\d$ is even and $n=2m$ is even.

We have $N_\tau=[(n+1)(n+\s+1)+1]/2-\tau$ and
\begin{align*}
  M&={(n+\d+1)(n+\d+2)\over2}-{(n+1)(n+\s+1)+1\over2}+\tau
  \\&={\d(\d+1)\over2}+m+\tau
\end{align*}
for $\tau=0,1$.
We define
\begin{equation}\label{Q3}
  Q=V+\spa\{p_{n-j}(x)q_{j+\d}(y)-(-1)^\tau p_j(x)q_{n+\d-j}(y),~~0\le j\le m-1+\tau\},
\end{equation}
which is a subspace of bivariate vanishing polynomials in $\P_{n+\d}(x,y)$ with dimension $M=\d(\d+1)/2+m+\tau$.
\end{enumerate}
On account of Theorem \ref{thm-uniqueness}, we have the following uniqueness property of bivariate Lagrange basis polynomials for $S_\tau$ with $\tau=0$ or $\tau=1$.
\begin{thm}
  Given any $x_0>x_1>\cdots>x_n$ and $y_0>y_1>\cdots>y_{n+\s}$ where $n$ and $\s$ are nonnegative integers.
  Let $S_\tau$ with either $\tau=0$ or $\tau=1$ be defined in \eqref{S0} or \eqref{S1}.
  Set $\d=\lfloor\s/2\rfloor$.
  The set of bivariate Lagrange basis polynomials for $S_\tau$ defined in \eqref{L} is unique in the quotient space $\P_{n+\d}(x,y)/Q$, where $Q$ is defined in \eqref{Q1}-\eqref{Q3} depending on odd-even properties of $\s$ and $n$.
\end{thm}

\section*{Acknowledgment}
LC is partially supported by National Natural Science Foundation of China (No. 11571375), the Natural Science Funding of Shenzhen University (No. 2018073), and the Shenzhen Scientific Research and Development Funding Program (No. JCYJ20170302144002028).

\end{document}